\documentclass[12pt]{amsart}

\usepackage{mathrsfs,amssymb}

\newtheorem{theorem}{Theorem}[section]
\newtheorem{lemma}[theorem]{Lemma}

\newtheorem{cor}[theorem]{Corollary}

\theoremstyle{definition}
\newtheorem{definition}[theorem]{Definition}
\newtheorem{que}[theorem]{Question}

\theoremstyle{remark}

\numberwithin{equation}{section}

\let \la=\lambda
\let \e=\varepsilon
\let \d=\delta

\let \a=\alpha

\let \b=\beta

\let \ga=\gamma

\begin{document}

\title[On a dual property of the maximal operator]
{On a dual property of the maximal operator on weighted variable $L^p$ spaces}

\author{Andrei K. Lerner}
\address{Department of Mathematics,
Bar-Ilan University, 5290002 Ramat Gan, Israel}
\email{lernera@math.biu.ac.il}

\thanks{This research was supported by the Israel Science Foundation (grant No. 953/13).}

\begin{abstract}
L. Diening \cite{D1} obtained the following dual property of the
maximal operator $M$ on variable Lebesque spaces $L^{p(\cdot)}$: if $M$ is bounded on $L^{p(\cdot)}$, then
$M$ is bounded on $L^{p'(\cdot)}$.
We extend this result to weighted variable Lebesque spaces.
\end{abstract}


\maketitle

\section{Introduction}

Given a measurable function $p:{\mathbb R}^n\to [1,\infty)$,
denote by $L^{p(\cdot)}$ the space of functions $f$ such
that for some $\la>0$,
$$\int_{{\mathbb R}^n}|f(x)/\la|^{p(x)}dx<\infty,$$
with norm
$$\|f\|_{L^{p(\cdot)}}=\inf\left\{\la>0:\int_{{\mathbb
R}^n}|f(x)/\la|^{p(x)}dx\le 1\right\}.$$
Set
$p_-\equiv\displaystyle \operatornamewithlimits{ess\,
inf}_{x\in\mathbb{R}^n} p(x)$ and $p_+\equiv\displaystyle
\operatornamewithlimits{ess\, sup}_{x\in\mathbb{R}^n} p(x)$.

Let $M$ be the Hardy-Littlewood maximal operator defined by
$$Mf(x)=\sup_{Q\ni x}\frac{1}{|Q|}\int_Q|f(y)|dy,$$
where the supremum is taken over all cubes $Q\subset {\mathbb R}^n$ containing the point $x$.

In \cite{D1}, L. Diening proved the following remarkable result: if $p_->~1$, $p_+<\infty$ and $M$ is bounded on $L^{p(\cdot)}$,
then $M$ is bounded on $L^{p'(\cdot)}$, where $p'(x)=\frac{p(x)}{p(x)-1}$.
Despite its apparent simplicity, the proof in \cite{D1}
is rather long and involved.

In this paper we extend Diening's theorem to weighted variable Lebesgue spaces $L^{p(\cdot)}_w$
equipped with norm
$$\|f\|_{L^{p(\cdot)}_w}=\|fw\|_{L^{p(\cdot)}}.$$
We assume that a weight $w$ here is a non-negative function such that $w(\cdot)^{p(\cdot)}$ and $w(\cdot)^{-p'(\cdot)}$
are locally integrable.
The spaces $L^{p(\cdot)}_w$ have been studied in numerous works; we refer to the monographs \cite{CF, DHHR} for a detailed
bibliography.

Recall that a non-negative locally integrable function $v$ satisfies the Muckenhoupt $A_r, 1<r<\infty,$ condition if
$$\sup_{Q}\left(\frac{1}{|Q|}\int_Qv\,dx\right)\left(\frac{1}{|Q|}\int_Qv^{-\frac{1}{r-1}}\,dx\right)^{r-1}<\infty.$$
Set $A_{\infty}=\cup_{r>1}A_r$.

Our main result is the following.

\begin{theorem}\label{mr}
Let $p:{\mathbb R}^n\to [1,\infty)$ be a measurable function such that $p_->1$ and $p_+<\infty$. Let $w$ be a weight such that $w(\cdot)^{p(\cdot)}\in A_{\infty}$.
If $M$ is bounded on $L^{p(\cdot)}_w$, then $M$ is bounded on $L^{p'(\cdot)}_{w^{-1}}$.
\end{theorem}

The relevance of the condition $w(\cdot)^{p(\cdot)}\in A_{\infty}$ in this theorem will be discussed in Section 6 below.

Notice that $L^{p'(\cdot)}_{w^{-1}}$ is the associate space of $L^{p(\cdot)}_w$, namely, $\big(L^{p(\cdot)}_w\big)'=L^{p'(\cdot)}_{w^{-1}}$
(see Sections 2.1 and 2.2). Hence, it is desirable to characterize Banach function spaces $X$ with the property that the boundedness of $M$ on $X$ implies
the boundedness of $M$ on $X'$. In Section~3, we obtain such a characterization in terms of an $A_{\infty}$-type property of~$X$. However, a verification of this
property in the case of $X=L^{p(\cdot)}_w$ is not as simple. In doing so, we use some ingredients developed by L.~Diening in \cite{D1} (Lemmas \ref{rh} and \ref{strf}).
We slightly simplified their proofs and we give them here in order to keep the paper essentially self-contained.

\section{Preliminaries}

\subsection{Banach function spaces}
Denote by ${\mathcal M}^+$ the set of Lebesgue measurable non-negative functions on ${\mathbb R}^n$.

\begin{definition}\label{BFS}
By a Banach function space (BFS) $X$ over ${\mathbb R}^n$ equipped with Lebesque measure we mean a collection of functions $f$ such
that
$$\|f\|_{X}=\rho(|f|)<\infty,$$
where $\rho: {\mathcal M}^+\to [0,\infty]$ is a mapping satisfying
\begin{enumerate}
\renewcommand{\labelenumi}{(\roman{enumi})}
\item
$\rho(f)=0\Leftrightarrow f=0$ a.e.; $\rho(\a f)=\a \rho(f), \a\ge 0$;

\noindent
$\rho(f+g)\le \rho(f)+\rho(g)$;
\item $g\le f\,\,\text{a.e.}\,\,\Rightarrow \rho(g)\le \rho(f)$;
\item $f_n\uparrow f \,\,\text{a.e.}\,\,\Rightarrow \rho(f_n)\uparrow \rho(f)$;
\item if $E\subset {\mathbb R}^n$ is bounded, then $\rho(\chi_E)<\infty$;
\item if $E\subset {\mathbb R}^n$ is bounded, then $\int_Efdx\le c_E\rho(f)$.
\end{enumerate}
\end{definition}

Note that it is more common to require that $E$ is a set of finite measure in (iv) and (v) (see, e.g., \cite{BS}). However, our
choice of axioms allows us to include weighted variable Lebesque spaces $L^{p(\cdot)}_w$ (with the assumption
that $w(\cdot)^{p(\cdot)}, w(\cdot)^{-p'(\cdot)}\in L^1_{loc}$) in a general framework of Banach function spaces.
Moreover, it is well known that all main elements of a general theory work with (iv) and (v) stated for bounded sets
(see, e.g., \cite{Lux}). We mention only the next two key properties that are of interest for us.
The first property says that if $X$ is a BFS, then the associate space $X'$ consisting of $f$ such that
$$\|f\|_{X'}=\sup_{g\in X:\|g\|_{X}\le 1}\int_{{\mathbb R}^n}|fg|\,dx<\infty$$
is also a BFS. The second property is the Lorentz-Luxemburg theorem saying that $X=X''$ and $\|f\|_X=\|f\|_{X''}$.

The definition of $\|f\|_{X'}$ implies that
\begin{equation}\label{hold}
\int_{{\mathbb R}^n}|fg|dx\le \|f\|_{X}\|g\|_{X'},
\end{equation}
and the fact that $\|f\|_X=\|f\|_{X''}$ yields
\begin{equation}\label{repr}
\|f\|_{X}=\sup_{g\in X':\|g\|_{X'}\le 1}\int_{{\mathbb R}^n}|fg|\,dx.
\end{equation}

\subsection{Variable $L^p$ spaces}
It is well known (see \cite{CF} or \cite{DHHR}) that if $p:{\mathbb R}^n\to [1,\infty)$, then $L^{p(\cdot)}$
is a BFS. Further, if $p_->1$ and $p_+<\infty$, then $(L^{p(\cdot)})'=L^{p'(\cdot)}$ and
\begin{equation}\label{as}
\frac{1}{2}\|f\|_{L^{p'(\cdot)}}\le \|f\|_{(L^{p(\cdot)})'}\le 2\|f\|_{L^{p'(\cdot)}}
\end{equation}
(see \cite[p. 78]{DHHR}).

Assume now that $p:{\mathbb R}^n\to [1,\infty)$ and $w$ is a weight such that $w(\cdot)^{p(\cdot)}$ and $w(\cdot)^{-p'(\cdot)}$
are locally integrable.
The weighted space $L^{p(\cdot)}_w$ consists of all $f$ such that
$$\|f\|_{L^{p(\cdot)}_w}=\|fw\|_{L^{p(\cdot)}}<\infty.$$

It is easy to see that $L^{p(\cdot)}_w$ is a BFS. Indeed, axioms (i)-(iii) of Definition~\ref{BFS} follow immediately
from the fact that the unweighted $L^{p(\cdot)}$ is a BFS. Next, (iv) follows from that $w(\cdot)^{p(\cdot)}\in L^1_{loc}$.
Finally, applying (\ref{hold}) with $X=L^{p(\cdot)}$ along with (\ref{as}) yields
$$\int_Efdx\le 2\|fw\|_{L^{p(\cdot)}}\|w^{-1}\chi_E\|_{L^{p'(\cdot)}},$$
and this proves (v) with $c_E=2\|w^{-1}\chi_E\|_{L^{p'(\cdot)}}<\infty$
(here we have used that $w(\cdot)^{-p'(\cdot)}\in L^1_{loc}$).

Since $\|fw^{-1}\|_{\big(L^{p(\cdot)}\big)'}=\|f\|_{\big(L^{p(\cdot)}_w\big)'}$, we obtain from (\ref{as}) that
if $p_->1$ and $p_+<\infty$, then $\big(L^{p(\cdot)}_w\big)'=L^{p'(\cdot)}_{w^{-1}}$ and
$$
\frac{1}{2}\|f\|_{L^{p'(\cdot)}_{w^{-1}}}\le \|f\|_{\big(L^{p(\cdot)}_w\big)'}\le 2\|f\|_{L^{p'(\cdot)}_{w^{-1}}}.
$$

Denote $\varrho(f)=\int_{{\mathbb R}^n}|f(x)|^{p(x)}dx$. We will frequently use the following lemma (see \cite[p. 25]{CF}).
\begin{lemma}\label{modest} Let $p:{\mathbb R}^n\to [1,\infty)$ and $p_+<\infty$. If $\|f\|_{L^{p(\cdot)}}>1$, then
$$\varrho(f)^{1/p_+}\le \|f\|_{L^{p(\cdot)}}\le \varrho(f)^{1/p_-}.$$
If $\|f\|_{L^{p(\cdot)}}\le 1$, then
$$\varrho(f)^{1/p_-}\le \|f\|_{L^{p(\cdot)}}\le \varrho(f)^{1/p_+}.$$
\end{lemma}

\subsection{Dyadic grids and sparse families}
The standard dyadic grid in ${\mathbb R}^n$ consists of the cubes
$$2^{-k}([0,1)^n+j),\quad k\in{\mathbb Z}, j\in{\mathbb Z}^n.$$
Following its basic properties, we say that a family of cubes ${\mathscr{D}}$
is a general dyadic grid if (i) for any $Q\in {\mathscr{D}}$ its sidelength $\ell_Q$ is of the form
$2^k, k\in {\mathbb Z}$; (ii) $Q\cap R\in\{Q,R,\emptyset\}$ for any $Q,R\in {\mathscr{D}}$;
(iii) for every $k\in {\mathbb Z}$, the cubes of a fixed sidelength $2^k$ form a partition of ${\mathbb
R}^n$.

Given a dyadic grid ${\mathscr{D}}$, consider the associated dyadic maximal operator $M^{{\mathscr{D}}}$ defined by
$$M^{{\mathscr{D}}}f(x)=\sup_{Q\ni x, Q\in {\mathscr{D}}}\frac{1}{|Q|}\int_Q|f(y)|dy.$$
On one hand, it is clear that $M^{{\mathscr{D}}}f\le Mf$. However, this inequality can be reversed, in a sense, as
the following lemma shows (its proof can be found in \cite[Lemma 2.5]{HLP}).

\begin{lemma}\label{dyad} There are $3^n$ dyadic grids ${\mathscr{D}}_{\a}$ such that for every cube $Q\subset{\mathbb R}^n$, there exists a cube $Q_{\a}\in {\mathscr{D}}_{\a}$
such that $Q\subset Q_{\a}$ and $|Q_{\a}|\le 6^n|Q|$.
\end{lemma}

We obtain from this lemma that for all $x\in {\mathbb R}^n$,
\begin{equation}\label{interm}
Mf(x)\le 6^n\sum_{\a=1}^{3^n}M^{{\mathscr{D}}_{\a}}f(x).
\end{equation}

Given a cube $Q_0$, denote by ${\mathcal D}(Q_0)$ the set of all
dyadic cubes with respect to $Q_0$, that is, the cubes from ${\mathcal D}(Q_0)$ are formed
by repeated subdivision of $Q_0$ and each of its descendants into $2^n$ congruent subcubes.
Consider the local dyadic maximal operator $M_{Q_0}^{d}$ defined by
$$M_{Q_0}^{d}f(x)=\sup_{Q\ni x, Q\in {\mathcal D}(Q_0)}\frac{1}{|Q|}\int_Q|f(y)|dy.$$

Denote $f_Q=\frac{1}{|Q|}\int_Qf$. The following lemma is a standard variation of the Calder\'on-Zygmund decomposition (see, e.g., \cite[Theorem 4.3.1]{G1}).
We include its proof for the reader convenience.

\begin{lemma}\label{cz} Suppose ${\mathscr{D}}$ is a dyadic grid.
Let $f\in L^p({\mathbb R}^n), 1\le p<\infty$, and let $\ga>1$. Assume that
$$\Omega_k=\{x\in {\mathbb R}^n: M^{{\mathscr{D}}}f(x)>\ga^k\}\not=\emptyset\quad(k\in {\mathbb Z}).$$
Then $\Omega_k$ can be written as a union of pairwise disjoint cubes $Q_j^k\in {\mathscr{D}}$ satisfying
\begin{equation}\label{czprop}
|Q_j^k\cap \Omega_{k+l}|\le 2^n(1/\gamma)^l|Q_j^k|\quad(l\in {\mathbb Z}_+).
\end{equation}
The same property holds in the local case for the sets
$$\Omega_k=\{x\in Q_0: M_{Q_0}^{d}f(x)>\ga^k|f|_{Q_0}\}\quad(f\in L^1(Q_0), k\in {\mathbb Z}_+).$$
\end{lemma}

\begin{proof} Consider the case of ${\mathbb R}^n$, the same proof works in the local case.
Let $Q_j^k$ be the maximal cubes such that $|f|_{Q_j^k}>\ga^k$. Then, by maximality, they are
pairwise disjoint and $|f|_{Q_j^k}\le 2^n\ga^k$. Also, $\Omega_k=\cup_jQ_j^k$. Therefore,
$$
|Q_j^k\cap \Omega_{k+l}|=\sum_{Q^{k+l}_i\subset Q_j^k}|Q_i^{k+l}|<(1/\gamma)^{k+l}\int_{Q_j^k}|f|
\le 2^n(1/\gamma)^l|Q_j^k|.
$$
\end{proof}

\begin{definition}\label{sparse}
Let ${\mathscr{D}}$ be a dyadic grid, and let $0<\eta<1$. We say that a family of cubes ${\mathcal S}\subset {\mathscr{D}}$ is $\eta$-sparse if
for every cube $Q\in {\mathcal S}$, there is a measurable subset $E(Q)\subset Q$ such that $\eta|Q|\le |E(Q)|$ and the
sets $\{E(Q)\}_{Q\in {\mathcal S}}$ are pairwise disjoint.
\end{definition}

\begin{lemma}\label{CZ} Let ${\mathscr{D}}$ be a dyadic grid, and let $0<\eta<1$. For every non-negative $f\in L^p({\mathbb R}^n), 1\le p<\infty$,
there exists an $\eta$-sparse family ${\mathcal S}\subset {\mathscr{D}}$ such that for all $x\in {\mathbb R}^n$,
$$M^{{\mathscr{D}}}f(x)\le \frac{2^n}{1-\eta}\sum_{Q\in {\mathcal S}}f_Q\chi_{E(Q)}(x).$$
\end{lemma}

\begin{proof} For $k\in {\mathbb Z}$, set $\Omega_k=\Big\{x\in {\mathbb R}^n:M^{{\mathscr{D}}}f(x)>\Big(\frac{2^n}{1-\eta}\Big)^k\Big\}.$
Then, by Lemma \ref{cz}, $\Omega_k=\cup_jQ_j^k$, and $|Q_j^k\cap\Omega_{k+1}|\le (1-\eta)|Q_j^k|$.
Therefore, setting $E(Q_j^k)=Q_j^k\setminus \Omega_{k+1}$, we obtain that $\eta|Q_j^k|\le |E(Q_j^k)|$, and the sets $\{E(Q_j^k)\}$ are pairwise disjoint. Further,
\begin{eqnarray*}
M^{{\mathscr{D}}}f\le \sum_{k\in {\mathbb Z}}(M^{{\mathscr{D}}}f)\chi_{\Omega_k\setminus\Omega_{k+1}}&\le& \frac{2^n}{1-\eta}\sum_{k\in{\mathbb Z}}
\Big(\frac{2^n}{1-\eta}\Big)^k\chi_{\Omega_k\setminus\Omega_{k+1}}\\
&\le& \frac{2^n}{1-\eta}\sum_{j,k}f_{Q_j^k}\chi_{E(Q_j^k)},
\end{eqnarray*}
which completes the proof with ${\mathcal S}=\{Q_j^k\}$.
\end{proof}

\subsection{$A_p$ weights}
Given a weight $w$ and a measurable set $E\subset {\mathbb R}^n$, denote $w(E)=\int_Ewdx$.
Given an $A_p, 1<p<\infty,$ weight, its $A_p$ constant is defined by
$$[w]_{A_p}=\sup_Q\left(\frac{1}{|Q|}\int_Qwdx\right)\left(\frac{1}{|Q|}\int_Qw^{-1/(p-1)}dx\right)^{p-1}.$$

Every $A_p$ weight satisfies the reverse H\"older inequality (see, e.g., \cite[Theorem 9.2.2]{G}), namely, there exist
$c>0$ and $r>1$ such that for any cube~$Q$,
\begin{equation}\label{rhold}
\left(\frac{1}{|Q|}\int_Qw^rdx\right)^{1/r}\le c\frac{1}{|Q|}\int_Qw\,dx.
\end{equation}
It follows from this and from H\"older's inequality that for every $Q$ and any measurable subset $E\subset Q$,
\begin{equation}\label{Ainf}
\frac{w(E)}{w(Q)}\le c\left(\frac{|E|}{|Q|}\right)^{1/r'}.
\end{equation}

Notice also that the following converse estimate
\begin{equation}\label{conv}
\frac{w(Q)}{w(E)}\le \left(\frac{|Q|}{|E|}\right)^p[w]_{A_p}\quad (E\subset Q, |E|>0)
\end{equation}
holds for all $p>1$. Indeed, by H\"older's inequality,
$$|E|^p\le \left(\int_Ew\,dx\right)\left(\int_Ew^{-1/(p-1)}\,dx\right)^{p-1},$$
which along with the definition of $[w]_{A_p}$ implies (\ref{conv}).

\section{Maximal operator on associate spaces}
Since $\big(L^{p(\cdot)}_w\big)'=L^{p'(\cdot)}_{w^{-1}}$, the statement of Theorem \ref{mr} leads naturally to a question about
conditions on a BFS $X$ such that $M:X\to X\Rightarrow M:X'\to X'$. The result below provides a criterion in terms of sparse
families and an $A_{\infty}$-type condition. Its proof is based essentially on the theory of $A_p$ weights.

\begin{theorem}\label{char} Let $X$ be a BFS such that
the Hardy-Littlewood maximal operator $M$ is bounded on $X$. Let $0<\eta<1$. The following conditions are equivalent:
\begin{enumerate}
\renewcommand{\labelenumi}{(\roman{enumi})}
\item
$M$ is bounded on $X'$;
\item
there exist $c,\d>0$ such that for every dyadic grid ${\mathscr D}$ and any finite $\eta$-sparse family ${\mathcal S}\subset {\mathscr D}$,
$$\Big\|\sum_{Q\in {\mathcal S}}\a_Q\chi_{G_Q}\Big\|_X\le c\left(\max_{Q\in {\mathcal S}}\frac{|G_Q|}{|Q|}\right)^{\d}\Big\|\sum_{Q\in {\mathcal S}}\a_Q\chi_{Q}\Big\|_X,$$
where $\{\a_Q\}_{Q\in {\mathcal S}}$ is an arbitrary sequence of non-negative numbers, and $\{G_Q\}_{Q\in {\mathcal S}}$ is any sequence of pairwise disjoint measurable subsets $G_Q\subset Q$.
\end{enumerate}
\end{theorem}

\begin{proof} Let us first prove $(\rm{i})\Rightarrow (ii)$. Let $g\ge 0$ and $\|g\|_{X'}=1$.
We use the standard Rubio de Francia algorithm \cite{R}, namely, set
$$Rg=\sum_{k=0}^{\infty}\frac{M^kg}{(2\|M\|_{X'})^k},$$
where $M^k$ denotes the $k$-th iteration of $M$ and $M^0g=g$. Then $g\le Rg$ and $\|Rg\|_{X'}\le 2$. Also,
$$M(Rg)(x)\le 2\|M\|_{X'}Rg(x).$$
Therefore, $Rg\in A_1$.

Using the properties of $Rg$ along with (\ref{Ainf}) and H\"olders inequality~(\ref{hold}), we obtain that there exist $c,\d>0$ such that
\begin{eqnarray*}
\int_{{\mathbb R}^n}\Big(\sum_{Q\in {\mathcal S}}\a_Q\chi_{G_Q}\Big)g\,dx&\le&\sum_{Q\in {\mathcal S}}\a_Q\int_{{G_Q}}Rg\,dx\\
&\le& c\sum_{Q\in {\mathcal S}}\a_Q\left(\frac{|G_Q|}{|Q|}\right)^{\d}\int_QRg\,dx\\
&\le& c\left(\max_{Q\in {\mathcal S}}\frac{|G_Q|}{|Q|}\right)^{\d}\int_{{\mathbb R}^n}\Big(\sum_{Q\in {\mathcal S}}\a_Q\chi_{Q}\Big)Rg\,dx\\
&\le& 2c\left(\max_{Q\in {\mathcal S}}\frac{|G_Q|}{|Q|}\right)^{\d}\Big\|\sum_{Q\in {\mathcal S}}\a_Q\chi_{Q}\Big\|_X.
\end{eqnarray*}
It remains to take here the supremum over all $g\ge 0$ with $\|g\|_{X'}=1$ and to use (\ref{repr}).

Turn to the proof of $(\rm{ii})\Rightarrow (i)$.
By (\ref{interm}), it suffices to prove that the dyadic maximal operator $M^{\mathscr D}$ is bounded on $X'$.
Let us show that there is $c>0$ such that for every $f\in L^1\cap X'$,
\begin{equation}\label{mdx}
\|M^{\mathscr D}f\|_{X'}\le c\|f\|_{X'}.
\end{equation}
Notice that (\ref{mdx}) implies the boundedness of $M^{\mathscr D}$ on $X'$. Indeed, having (\ref{mdx}) established,
for an arbitrary $f\in X'$ we apply (\ref{mdx}) to $f_N=f\chi_{\{|x|\le N\}}$ (clearly, $f_N\in L^1\cap X'$).
Letting then $N\to\infty$ and using the Fatou property ((iii) of Definition \ref{BFS}), we obtain that (\ref{mdx}) holds for any $f\in X'$.

In order to prove (\ref{mdx}), by Lemma \ref{CZ}, it suffices to show that the operator
$${\mathcal M}_{\mathcal S}f=\sum_{Q\in {\mathcal S}}f_Q\chi_{E(Q)}$$
satisfies
$$\|{\mathcal M}_{\mathcal S}f\|_{X'}\le c\|f\|_{X'}$$
for every non-negative $f\in L^1\cap X'$ with $c>0$ independent of $f$ and ${\mathcal S}$.
Notice that here ${\mathcal S}=\{Q_j^k\}$, and $Q_j^k$ are maximal dyadic cubes forming the set
$$\Omega_k=\Big\{x\in {\mathbb R}^n:M^{{\mathscr{D}}}f(x)>\Big(\frac{2^n}{1-\eta}\Big)^k\Big\}.$$

By duality, it is enough to obtain the uniform boundedness of the adjoint operator
$${\mathcal M}_{\mathcal S}^{\star}f=\sum_{j,k}\left(\frac{1}{|Q_j^k|}\int_{E(Q_j^k)}f\right)\chi_{Q_j^k}$$
on $X$. Using the Fatou property again, one can assume that ${\mathcal S}$ is finite.

Take $\nu\in {\mathbb N}$ such that
$$2^{n\d}c\sum_{l=\nu}^{\infty}\Big(\frac{1-\eta}{2^n}\Big)^{\l\d}\le \frac{1}{2},$$
where $c$ and $\d$ are the constants from condition (ii). Denote $\a_{j,k}=\frac{1}{|Q_j^k|}\int_{E(Q_j^k)}f$.
Then, using that $\cup_jQ_j^k\setminus\Omega_{k+\nu}=\cup_{i=0}^{\nu-1}\Omega_{k+i}\setminus\Omega_{k+i+1}$, we obtain
\begin{eqnarray*}
{\mathcal M}_{\mathcal S}^{\star}f&\le& \sum_{j,k}\a_{j,k}\chi_{Q_j^k\setminus\Omega_{k+\nu}}+\sum_{j,k}\a_{j,k}\chi_{Q_j^k\cap\Omega_{k+\nu}}\\
&\le& \nu Mf+\sum_{l=\nu}^{\infty}\sum_{j,k}\a_{j,k}\chi_{Q_j^k\cap(\Omega_{k+l}\setminus\Omega_{k+l+1})}.
\end{eqnarray*}
Therefore, applying (\ref{czprop}) along with condition (ii), we obtain
\begin{eqnarray*}
\|{\mathcal M}_{\mathcal S}^{\star}f\|_{X}&\le& \nu\|Mf\|_{X}+\sum_{l=\nu}^{\infty}\|\sum_{j,k}\a_{j,k}\chi_{Q_j^k\cap(\Omega_{k+l}\setminus\Omega_{k+l+1})}\|_{X}\\
&\le& \nu\|M\|_{X}\|f\|_{X}+2^{n\d}c\sum_{l=\nu}^{\infty}\Big(\frac{1-\eta}{2^n}\Big)^{\l\d}\|\sum_{j,k}\a_{j,k}\chi_{Q_j^k}\|_{X}\\
&\le& \nu\|M\|_{X}\|f\|_{X}+\frac{1}{2}\|{\mathcal M}_{\mathcal S}^{\star}f\|_{X}.
\end{eqnarray*}
Since ${\mathcal S}$ is finite, by (iv) of Definition \ref{BFS} we obtain that $\|{\mathcal M}_{\mathcal S}^{\star}f\|_{X}<\infty$.
Hence,
$$\|{\mathcal M}_{\mathcal S}^{\star}f\|_{X}\le 2\nu\|M\|_{X}\|f\|_{X},$$
and this completes the proof of $(\rm{ii})\Rightarrow (i)$.
\end{proof}

\section{Proof of Theorem \ref{mr}}
Take $X=L^{p(\cdot)}_w$ in Theorem \ref{char}. All we have to do is to check condition (ii) in this theorem. In order to do that, we need
a kind of the reverse H\"older property for the weights $(tw(x))^{p(x)}$. The following key lemma provides a replacement of such a property
which is enough for our purposes.

\begin{lemma}\label{key}
Let $1<p_-\le p_+<\infty$. Assume that $w(\cdot)^{p(\cdot)}\in A_{\infty}$ and that $M$ is bounded on~$L^{p(\cdot)}_w$.
Then there exist $\ga>1$ and $c,\eta>0$, and there is a measure $b$ on ${\mathbb R}^n$
such that for every cube $Q$ and all $t>0$ such that $t\|\chi_Q\|_{L^{p(\cdot)}_w}\le 1$ one has
\begin{eqnarray}
&&|Q|\left(\frac{1}{|Q|}\int_Q(tw(x))^{\ga p(x)}dx\right)^{1/\ga}\label{maincond}\\
&&\le c\int_Q(tw(x))^{p(x)}dx+2t^{\eta}b(Q)\chi_{(0,1)}(t),\nonumber
\end{eqnarray}
and for every finite family of pairwise disjoint cubes $\pi$, $\sum_{Q\in \pi}b(Q)\le c.$
\end{lemma}

The proof of this lemma is rather technical, and we postpone it until the next Section. Let us see now how the proof of Theorem \ref{mr} follows.

\begin{proof}[Proof of Theorem \ref{mr}]
Let ${\mathscr D}$ be a dyadic grid, and let ${\mathcal S}\subset {\mathscr D}$ be a finite $\frac{1}{2}$-sparse family. Let $\{G_Q\}_{Q\in {\mathcal S}}$
be a family of pairwise disjoint sets such that $G_Q\subset Q$. Take any sequence of non-negative numbers $\{\a_Q\}_{Q\in {\mathcal S}}$ such that
\begin{equation}\label{eqone}
\Big\|\sum_{Q\in {\mathcal S}}\a_Q\chi_Q\Big\|_{L^{p(\cdot)}_w}=1.
\end{equation}
By Lemma \ref{modest} and Theorem \ref{char}, it suffices to show that there exist absolute constants $c,\d>0$ such that
\begin{equation}\label{suff}
\sum_{Q\in {\mathcal S}}\int_{G_Q}(\a_Qw(x))^{p(x)}dx\le c\left(\max_{Q\in {\mathcal S}}\frac{|G_Q|}{|Q|}\right)^{\d}.
\end{equation}

It follows from (\ref{eqone}) that $\a_Q\|\chi_Q\|_{L^{p(\cdot)}_w}\le 1$ for every $Q\in {\mathcal S}$.
Therefore, if $\a_Q\ge 1$, by Lemma \ref{key} and H\"older's inequality along with (\ref{eqone}) we obtain
\begin{eqnarray}
&&\sum_{Q\in {\mathcal S}:\a_Q\ge 1}\int_{G_Q}(\a_Qw(x))^{p(x)}dx\label{bigone}\\
&&\le \sum_{Q\in {\mathcal S}:\a_Q\ge 1}|Q|\left(\frac{|G_Q|}{|Q|}\right)^{1/\ga'}\left(\frac{1}{|Q|}\int_Q(\a_Qw(x))^{\ga p(x)}dx\right)^{1/\ga}\nonumber\\
&&\le c\sum_{Q\in {\mathcal S}:\a_Q\ge 1}\left(\frac{|G_Q|}{|Q|}\right)^{1/\ga'}\int_Q(\a_Qw(x))^{p(x)}dx\le c\left(\max_{Q\in {\mathcal S}}\frac{|G_Q|}{|Q|}\right)^{1/\ga'}.\nonumber
\end{eqnarray}

The case when $\a_Q<1$ is more complicated because of the additional term on the right-hand side of (\ref{maincond}). We proceed as follows.
Denote
$$
{\mathcal S}_k=\{Q\in {\mathcal S}: 2^{-k}\le \a_Q< 2^{-k+1}\}\quad(k\in {\mathbb N}).
$$
Let $Q_i^k$ be the maximal cubes from ${\mathcal S}_k$ such that every other cube $Q\in {\mathcal S}_k$ is contained in one of them.
Then the cubes $Q_i^k$ are pairwise disjoint (for $k$ fixed). Set
$$\psi_{Q_i^k}(x)=\sum_{Q\in {\mathcal S}_k:Q\subseteq Q_i^k}\chi_{G_Q}(x).$$
Then
\begin{eqnarray*}
&&\sum_{Q\in {\mathcal S}:\a_Q<1}\int_{G_Q}(\a_Qw(x))^{p(x)}dx=\sum_{k=1}^{\infty}\sum_{Q\in {\mathcal S}_k}\int_{G_Q}(\a_Qw(x))^{p(x)}dx\\
&&=\sum_{k=1}^{\infty}\sum_i\sum_{Q\in {\mathcal S}_k:Q\subseteq Q_i^k}\int_{G_Q}(\a_Qw(x))^{p(x)}dx\\
&&\le 2^{p_+}\sum_{i,k}\int_{Q_i^k}(\a_{Q_i^k}w(x))^{p(x)}\psi_{Q_i^k}(x)dx.
\end{eqnarray*}

By H\"older's inequality,
\begin{eqnarray*}
&&\sum_{i,k}\int_{Q_i^k}(\a_{Q_i^k}w(x))^{p(x)}\psi_{Q_i^k}(x)dx\\
&&\le \sum_{i,k}|Q_i^k|\left(\frac{1}{|Q_i^k|}\int_{Q_i^k}(\a_{Q_i^k}w(x))^{\ga p(x)}dx\right)^{1/\ga}\left(\frac{1}{|Q_i^k|}\int_{Q_i^k}\psi_{Q_i^k}(x)^{\ga'}dx\right)^{1/\ga'}.
\end{eqnarray*}
Since ${\mathcal S}$ is $\frac{1}{2}$-sparse,
\begin{eqnarray*}
\int_{Q_i^k}\psi_{Q_i^k}(x)^{\ga'}dx=\sum_{Q\in {\mathcal S}_k:Q\subseteq Q_i^k}|G_Q|&\le&
\left(\max_{Q\in {\mathcal S}}\frac{|G_Q|}{|Q|}\right)\sum_{Q\in {\mathcal S}_k:Q\subseteq Q_i^k}|Q|\\
&\le& 2\left(\max_{Q\in {\mathcal S}}\frac{|G_Q|}{|Q|}\right)|Q_i^k|.
\end{eqnarray*}
Combining this with the two previous estimates yields
\begin{eqnarray*}
&&\sum_{Q\in {\mathcal S}:\a_Q<1}\int_{G_Q}(\a_Qw(x))^{p(x)}dx\\
&&\le c\left(\max_{Q\in {\mathcal S}}\frac{|G_Q|}{|Q|}\right)^{1/\ga'}
\sum_{i,k}|Q_i^k|\left(\frac{1}{|Q_i^k|}\int_{Q_i^k}(\a_{Q_i^k}w(x))^{\ga p(x)}dx\right)^{1/\ga}.
\end{eqnarray*}

By Lemma \ref{key} along with (\ref{eqone}), and Lemma \ref{modest},
\begin{eqnarray*}
\sum_{i,k}|Q_i^k|\left(\frac{1}{|Q_i^k|}\int_{Q_i^k}(\a_{Q_i^k}w(x))^{\ga p(x)}dx\right)^{1/\ga}\le c+2\sum_{i,k}\a_{Q_i^k}^{\eta}b(Q_i^k).
\end{eqnarray*}
Since for every fixed $k$, the cubes $\{Q_i^k\}$ are pairwise disjoint,
$$\sum_{i,k}\a_{Q_i^k}^{\eta}b(Q_i^k)\le 2^{\eta}\sum_{k=1}^{\infty}2^{-k\eta}\sum_{i}b(Q_i^k)\le c\sum_{k=1}^{\infty}2^{-k\eta}\le c.$$
This, combined with the two previous estimates implies
$$
\sum_{Q\in {\mathcal S}:\a_Q<1}\int_{G_Q}(\a_Qw(x))^{p(x)}dx\le c\left(\max_{Q\in {\mathcal S}}\frac{|G_Q|}{|Q|}\right)^{1/\ga'},
$$
which along with (\ref{bigone}) proves (\ref{suff}).
\end{proof}

\section{Proof of Lemma \ref{key}}
We split the proof of Lemma \ref{key} into several pieces.
Lemmas \ref{rh} and~\ref{strf} below are due to L. Diening \cite{D1}. We give slightly shortened versions of their proofs for the sake of completeness.
Notice that these lemmas hold for arbitrary weights $w$ such that $w(\cdot)^{p(\cdot)}$ is locally integrable.
Lemma \ref{RH} is new. The assumption that $w(\cdot)^{p(\cdot)}\in A_{\infty}$ is essential there. Throughout this section, we assume that
$p_->1$ and $p_+<\infty$.

\begin{lemma}\label{rh}
Assume that $M$ is bounded on $L^{p(\cdot)}_w$.
Then there exist $r,c>1$ such that for every family of pairwise disjoint cubes $\pi$
and for every sequence of non-negative numbers $\{t_Q\}_{Q\in\pi}$,
$$\sum_{Q\in \pi}\int_Q(t_Qw(x))^{p(x)}dx\le 1\Rightarrow \sum_{Q\in \pi}|Q|\left(\frac{1}{|Q|}\int_Q(t_Qw(x))^{rp(x)}dx\right)^{1/r}\le c.$$
\end{lemma}

\begin{proof}
Given a family $\pi$ and a sequence $\{t_Q\}_{Q\in \pi}$, denote $v_Q(x)=(t_Qw(x))^{p(x)}$ and $\a_Q=\frac{1}{|Q|}\int_Qv_Qdx$.

By Lemma \ref{cz}, write the set
$$\Omega_k(Q)=\{x\in Q:M^d_Qv_Q(x)>(2^{n+1})^{k}\a_Q\}\quad(k\in {\mathbb N})$$
as a union of  pairwise disjoint cubes $P_j^k(Q)$ satisfying $|E_j^k(Q)|\ge \frac{1}{2}|P_j^k(Q)|$, where
$E_j^k(Q)=P_j^k(Q)\setminus \Omega_{k+1}(Q)$. From this,
$$\sum_{Q\in \pi}t_Q\chi_{\Omega_k(Q)}\le 2M\big(\sum_{Q\in \pi}t_Q\chi_{\Omega_k(Q)\setminus\Omega_{k+1}(Q)}\big),$$
and hence,
$$\Big\|\sum_{Q\in \pi}t_Q\chi_{\Omega_k(Q)}\Big\|_{L^{p(\cdot)}_w}\le 2\|M\|_{L^{p(\cdot)}_w}\Big\|\sum_{Q\in \pi}t_Q\chi_{\Omega_k(Q)\setminus\Omega_{k+1}(Q)}\Big\|_{L^{p(\cdot)}_w}.$$

Setting $t_Q'=\frac{t_Q}{\|\sum_{Q\in \pi}t_Q\chi_{\Omega_k(Q)}\|_{L^{p(\cdot)}_w}}$, this inequality yields
$$1\le 2\|M\|_{L^{p(\cdot)}_w}\Big\|\sum_{Q\in \pi}t_Q'\chi_{\Omega_k(Q)\setminus\Omega_{k+1}(Q)}\Big\|_{L^{p(\cdot)}_w}.$$
Since 
$$\Big\|\sum_{Q\in \pi}t_Q'\chi_{\Omega_k(Q)\setminus\Omega_{k+1}(Q)}\Big\|_{L^{p(\cdot)}_w}\le 
\Big\|\sum_{Q\in \pi}t_Q'\chi_{\Omega_k(Q)}\Big\|_{L^{p(\cdot)}_w}=1,$$
Lemma \ref{modest} along with the previous estimate implies,
\begin{eqnarray*}
\frac{1}{(2\|M\|_{L^{p(\cdot)}_w})^{p_+}}&\le& \sum_{Q\in \pi}\int_{\Omega_k(Q)\setminus\Omega_{k+1}(Q)}(t'_Qw(x))^{p(x)}dx\\
&\le& 1-\sum_{Q\in \pi}\int_{\Omega_{k+1}(Q)}(t'_Qw(x))^{p(x)}dx,
\end{eqnarray*}
which in turn implies (again, by Lemma \ref{modest})
$$\Big\|\sum_{Q\in \pi}t'_Q\chi_{\Omega_{k+1}(Q)}\Big\|_{L^{p(\cdot)}_w}\le \b,$$
where $\b=\Big(1-\frac{1}{(2\|M\|_{L^{p(\cdot)}_w})^{p_+}}\Big)^{1/p_+}$.
Hence,
$$\Big\|\sum_{Q\in \pi}t_Q\chi_{\Omega_{k+1}(Q)}\Big\|_{L^{p(\cdot)}_w}\le \b
\Big\|\sum_{Q\in \pi}t_Q\chi_{\Omega_{k}(Q)}\Big\|_{L^{p(\cdot)}_w},$$
and thus, $\|\sum_{Q\in \pi}t_Q\chi_{\Omega_{k}(Q)}\|_{L^{p(\cdot)}_w}\le \b^{k-1},$ which by Lemma \ref{modest} implies
\begin{equation}\label{estwq}
\sum_{Q\in \pi}\int_{\Omega_{k}(Q)}(t_Qw(x))^{p(x)}dx\le \b^{{p_-}(k-1)}.
\end{equation}

Denote $\Omega_0(Q)=Q$. Then, for $\e>0$ we have
\begin{eqnarray*}
\int_Q(t_Qw(x))^{(1+\e)p(x)}dx&=&\sum_{k=0}^{\infty}\int_{\Omega_k(Q)\setminus\Omega_{k+1}(Q)}
(t_Qw(x))^{(1+\e)p(x)}dx\\
&\le& \a_Q^{\e}\sum_{k=0}^{\infty}2^{(n+1)(k+1)\e}\int_{\Omega_k(Q)}(t_Qw(x))^{p(x)}dx.
\end{eqnarray*}

Take $\e>0$ such that $\sum_{k=0}^{\infty}(2^{(n+1)\e}\b^{p_-})^k<\infty$. Then, combining the previous estimate with (\ref{estwq}) and H\"older's inequality, we
obtain
\begin{eqnarray*}
&&\sum_{Q\in \pi}|Q|\left(\frac{1}{|Q|}\int_Q(t_Qw(x))^{(1+\e)p(x)}dx\right)^{\frac{1}{1+\e}}\\
&&\le\sum_{Q\in \pi}v_Q(Q)^{\frac{\e}{1+\e}}\left(\sum_{k=0}^{\infty}2^{(n+1)(k+1)\e}v_Q(\Omega_k(Q))\right)^{\frac{1}{1+\e}}\\
&&\le \left(2^{n+1}\e+\sum_{k=1}^{\infty}2^{(n+1)(k+1)\e}\b^{{p_-}(k-1)}\right)^{\frac{1}{1+\e}}\le c,
\end{eqnarray*}
and therefore, the proof is complete.
\end{proof}

\begin{lemma}\label{strf}
Assume that $M$ is bounded on $L^{p(\cdot)}_w$. Then there exist $r,k>1$,
and a measure $b$ on ${\mathbb R}^n$ such that
the following properties hold: if $\int_Q(tw(x))^{p(x)}dx\le 1$, then
\begin{equation}\label{eqb}
|Q|\left(\frac{1}{|Q|}\int_Q(tw(x))^{rp(x)}dx\right)^{1/r}\le k\int_Q(tw(x))^{p(x)}dx+b(Q),
\end{equation}
and for every finite family of pairwise disjoint cubes $\pi$,
$\displaystyle\sum_{Q\in \pi}b(Q)\le 2k.$
\end{lemma}

\begin{proof} Let $r$ and $c$ be the constants from Lemma \ref{rh}. Set $k=2^{\frac{p_+}{p_-}+1}c$.

Given a cube $Q$, denote by $A(Q)$ the set of $t>0$ such that
$$\int_Q(tw(x))^{p(x)}dx\le~1$$
and
\begin{equation}\label{qleft}
|Q|\left(\frac{1}{|Q|}\int_Q(tw(x))^{rp(x)}dx\right)^{1/r}>k\int_Q(tw(x))^{p(x)}dx.
\end{equation}

Let $t_Q=\sup A(Q)$ (if $A(Q)=\emptyset$, set $t_Q=0$). Then
\begin{equation}\label{impcond}
\int_Q(t_Qw(x))^{p(x)}dx<1.
\end{equation}
Indeed, if $\int_Q(t_Qw(x))^{p(x)}dx=~1$,
we obtain
$$|Q|\left(\frac{1}{|Q|}\int_Q(t_Qw(x))^{rp(x)}dx\right)^{1/r}\ge k,$$
and this would contradict Lemma~\ref{rh}.
Further, we have
\begin{equation}\label{qleft1}
|Q|\left(\frac{1}{|Q|}\int_Q(t_Qw(x))^{rp(x)}dx\right)^{1/r}=k\int_Q(t_Qw(x))^{p(x)}dx,
\end{equation}
since otherwise (\ref{qleft}) holds with $t=t_Q$, and by continuity, using also (\ref{impcond}), we would obtain that $t_Q+\e\in A(Q)$ for some $\e>0$,
which contradicts the definition of $t_Q$.

Set now
$$b(Q)=|Q|\left(\frac{1}{|Q|}\int_Q(t_Qw(x))^{rp(x)}dx\right)^{1/r}.$$
Then (\ref{eqb}) holds trivially.

Let $\pi$ be any finite family of pairwise disjoint cubes. Let $\pi'\subseteq \pi$ be a maximal subset such that $\sum_{Q\in \pi'}\int_Q(t_Qw(x))^{p(x)}dx\le 2$
(maximal in the sense of the number of elements; this set is not necessarily unique, in general). We claim that $\pi'=\pi$. Indeed,
assume that $\pi'\not=\pi$.
Then we have $\sum_{Q\in \pi'}\int_Q(t_Qw(x)/2^{1/p_-})^{p(x)}dx\le 1,$ and by Lemma~\ref{rh},
$$\sum_{Q\in \pi'}|Q|\left(\frac{1}{|Q|}\int_Q(t_Qw(x))^{rp(x)}dx\right)^{1/r}\le 2^{\frac{p_+}{p_-}}c.$$
From this and from (\ref{qleft1}),
$$\sum_{Q\in \pi'}\int_Q(t_Qw(x))^{p(x)}dx=\frac{1}{k}\sum_{Q\in \pi'}|Q|\left(\frac{1}{|Q|}\int_Q(t_Qw(x))^{rp(x)}dx\right)^{1/r}\le \frac{1}{2}.$$
Therefore, if $P\in \pi\setminus \pi'$, we obtain
$$\sum_{Q\in \pi'\cup\{P\}}\int_Q(t_Qw(x))^{p(x)}dx\le \frac{3}{2},$$
which contradicts the maximality of $\pi'$. This proves that $\pi'=\pi$. Hence,
$$\sum_{Q\in \pi}b(Q)=k\sum_{Q\in \pi}\int_Q(t_Qw(x))^{p(x)}dx\le 2k,$$
which completes the proof.
\end{proof}

\begin{lemma}\label{RH}
Assume that $w(\cdot)^{p(\cdot)}\in A_{\infty}$ and that $M$ is bounded on~$L^{p(\cdot)}_w$. There exist $\ga,c>1$ and $\e>0$ such that
if
\begin{equation}\label{condit}
t\in \big[\min\big(1,1/\|\chi_Q\|^{1+\e}_{L^{p(\cdot)}_w}\big), \max\big(1,1/\|\chi_Q\|^{1+\e}_{L^{p(\cdot)}_w}\big)\big],
\end{equation}
then
\begin{equation}\label{revhol}
\left(\frac{1}{|Q|}\int_Q(tw(x))^{\ga p(x)}dx\right)^{1/\ga}\le c\frac{1}{|Q|}\int_Q(tw(x))^{p(x)}dx.
\end{equation}
\end{lemma}

\begin{proof} By the definition of $A_{\infty}$, there is an $s>1$ such that $w(\cdot)^{p(\cdot)}\in A_{s}$. By (\ref{rhold}), $w(\cdot)^{p(\cdot)}$ satisfies the reverse H\"older inequality 
with an exponent $\nu>1$. Let $r>1$ be the exponent from Lemma \ref{rh}. Take any $\ga$ satisfying $1<\ga<\min(\nu,r)$. Set
$\e=\frac{r-\ga}{\ga(1+(s-1)r)}$.

For every $\a>0$,
$$
\left(\frac{1}{|Q|}\int_Q(tw(x))^{\ga p(x)}dx\right)^{1/\ga}=\left(\frac{1}{|Q|}\int_Qt^{\ga(p(x)-\a)}w(x)^{\ga p(x)}dx\right)^{1/\ga}t^{\a}.
$$
Next, by (\ref{condit}), for all $x\in Q$,
$$t^{\ga(p(x)-\a)}\le 1+\|\chi_Q\|_{L^{p(\cdot)}_w}^{\ga\a(1+\e)}(1/\|\chi_Q\|^{1+\e}_{L^{p(\cdot)}_w})^{\ga p(x)},$$
and hence,
\begin{eqnarray*}
\int_Qt^{\ga(p(x)-\a)}w(x)^{\ga p(x)}dx&\le& \int_Qw(x)^{\ga p(x)}dx\\
&+&\|\chi_Q\|_{L^{p(\cdot)}_w}^{\ga\a(1+\e)}\int_Q\left(\frac{w(x)}{\|\chi_Q\|^{1+\e}_{L^{p(\cdot)}_w}}\right)^{\ga p(x)}dx.
\end{eqnarray*}
Combining this with the previous estimates yields
\begin{eqnarray}
&&\left(\frac{1}{|Q|}\int_Q(tw(x))^{\ga p(x)}dx\right)^{1/\ga}\le \left(\frac{1}{|Q|}\int_Qw(x)^{\ga p(x)}dx\right)^{1/\ga}t^{\a}\label{impstep}\\
&&+\|\chi_Q\|_{L^{p(\cdot)}_w}^{\a(1+\e)}
\left(\frac{1}{|Q|}\int_Q\left(\frac{w(x)}{\|\chi_Q\|^{1+\e}_{L^{p(\cdot)}_w}}\right)^{\ga p(x)}dx\right)^{1/\ga}t^{\a}.\nonumber
\end{eqnarray}

Let $\a=m_p(Q)$ be a median value of $p$ over $Q$, that is, a number satisfying
$$\max\left(\frac{|\{x\in Q:p(x)>m_p(Q)\}|}{|Q|}, \frac{|\{x\in Q:p(x)<m_p(Q)\}|}{|Q|}\right)\le \frac{1}{2}.$$
Set $E_1=\{x\in Q: p(x)\le m_p(Q)\}$ and $E_2=\{x\in Q: p(x)\ge m_p(Q)\}$.
Then $|E_1|\ge \frac{1}{2}|Q|$ and $|E_2|\ge \frac{1}{2}|Q|$.

Suppose, for instance, that $\|\chi_Q\|_{L^{p(\cdot)}_w}\le 1$. Then $t\ge 1$.
Let us estimate the first term on the right-hand side of (\ref{impstep}). Since $\ga<\nu$, the reverse H\"older inequality implies
$$\left(\frac{1}{|Q|}\int_Qw(x)^{\ga p(x)}dx\right)^{1/\ga}\le c\frac{1}{|Q|}\int_Qw(x)^{p(x)}dx.$$
By (\ref{conv}) and since $|E_2|\ge \frac{1}{2}|Q|$, 
$$\int_Qw(x)^{p(x)}dx\le c\int_{E_2}w(x)^{p(x)}dx.$$ 
Using also that $t\ge 1$, we obtain
\begin{eqnarray}
\left(\frac{1}{|Q|}\int_Qw(x)^{\ga p(x)}dx\right)^{1/\ga}t^{m_p(Q)}&\le& \frac{c}{|Q|}t^{m_p(Q)}\int_{E_2}w(x)^{p(x)}dx\label{estim1}\\
&\le& \frac{c}{|Q|}\int_Q(tw(x))^{p(x)}dx.\nonumber
\end{eqnarray}

Turn to the second term on the right-hand side of (\ref{impstep}).
The boundedness of $M$ on $L^{p(\cdot)}_w$ implies
$\|\chi_Q\|_{L^{p(\cdot)}_w}\le c\|\chi_{E_1}\|_{L^{p(\cdot)}_w}.$
By Lemma \ref{modest} (to be more precise, we use here a local version of Lemma \ref{modest}; see \cite[p. 25]{CF} for details),
$$
\|\chi_{E_1}\|_{L^{p(\cdot)}_w}\le \left(\int_{E_1}w(x)^{p(x)}dx\right)^{1/p_+(E_1)}
\le \left(\int_{E_1}w(x)^{p(x)}dx\right)^{1/m_p(Q)},
$$
where $p_+(E_1)=\displaystyle \operatornamewithlimits{ess\, sup}_{x\in E_1} p(x)$.
As previously, by (\ref{conv}),  $\int_{E_1}w(x)^{p(x)}dx\le c\int_{E_2}w(x)^{p(x)}dx$. Therefore, combining the previous
estimates yields
\begin{equation}\label{stepp}
\|\chi_Q\|_{L^{p(\cdot)}_w}\le c\left(\int_{E_2}w(x)^{p(x)}dx\right)^{1/m_p(Q)}.
\end{equation}

Let $q=\frac{1+r(s-1)}{1+\ga(s-1)}$ and $q'=\frac{q}{q-1}$. Then $q(1+\e)\ga=r$ and $q'\e\ga=\frac{1}{s-1}$.
Hence, H\"older's inequality with the exponents $q$ and $q'$ along with Lemma \ref{rh} implies
\begin{eqnarray*}
&&\frac{1}{|Q|}\int_Q\left(\frac{w(x)}{\|\chi_Q\|^{1+\e}_{L^{p(\cdot)}_w}}\right)^{\ga p(x)}dx\\
&&\le \left(\frac{1}{|Q|}\int_Q\left(\frac{w(x)}{\|\chi_Q\|_{L^{p(\cdot)}_w}}\right)^{rp(x)}dx\right)^{1/q}\left(\frac{1}{|Q|}\int_Qw(x)^{-\frac{1}{s-1}p(x)}dx\right)^{1/q'}\\
&&\le c\frac{1}{|Q|^{\frac{r-1}{q}+1}}\left(\int_Qw(x)^{-\frac{1}{s-1}p(x)}dx\right)^{1/q'}.
\end{eqnarray*}

Notice that
$$\frac{r-1}{q}+1=\frac{r}{q}+\frac{1}{q'}=(1+\e)\ga+\e\ga(s-1)=\ga(s\e+1).$$
Therefore, from the previous estimate and from (\ref{stepp}),
\begin{eqnarray*}
&&\|\chi_Q\|_{L^{p(\cdot)}_w}^{m_p(Q)\e}
\left(\frac{1}{|Q|}\int_Q\left(\frac{w(x)}{\|\chi_Q\|^{1+\e}_{L^{p(\cdot)}_w}}\right)^{\ga p(x)}dx\right)^{1/\ga}\\
&&\le c\frac{|Q|^{s\e}}{|Q|^{\big(\frac{r-1}{q}+1\big)\frac{1}{\ga}}}\left(\frac{1}{|Q|}\int_Qw(x)^{p(x)}dx\right)^{\e}\left(\frac{1}{|Q|}\int_Qw(x)^{-\frac{1}{s-1}p(x)}dx\right)^{\e(s-1)}\\
&&\le c[w(\cdot)^{p(\cdot)}]_{A_s}^{\e}\frac{1}{|Q|}.
\end{eqnarray*}
From this, using (\ref{stepp}) again, we obtain
\begin{eqnarray*}
&&\|\chi_Q\|_{L^{p(\cdot)}_w}^{m_p(Q)(1+\e)}
\left(\frac{1}{|Q|}\int_Q\left(\frac{w(x)}{\|\chi_Q\|^{1+\e}_{L^{p(\cdot)}_w}}\right)^{\ga p(x)}dx\right)^{1/\ga}t^{m_p(Q)}\\
&&\le \frac{c}{|Q|}t^{m_p(Q)}\int_{E_2}w(x)^{p(x)}dx\le \frac{c}{|Q|}\int_Q(tw(x))^{p(x)}dx.
\end{eqnarray*}
This along with (\ref{impstep}) and (\ref{estim1}) proves (\ref{revhol}).

Finally, we note that the proof in the case when $\|\chi_Q\|_{L^{p(\cdot)}_w}\ge 1$ is the same, with reversed roles of the sets $E_1$ and $E_2$.
\end{proof}

\begin{proof}[Proof of Lemma \ref{key}]
Assume that $t\|\chi_Q\|_{L^{p(\cdot)}_w}\le 1$. If $t\ge 1$, then the conclusion of Lemma \ref{key} follows immediately from Lemma \ref{RH}.
Therefore, it remains to consider the case when $t<1$.

We may keep all main settings of Lemma \ref{RH}, namely, assume that $w(\cdot)^{p(\cdot)}\in A_s$, and take the same numbers
$\ga$ and $\e=\frac{r-\ga}{\ga(1+(s-1)r)}$.

If
$$
\left(\frac{1}{|Q|}\int_Q(tw(x))^{\ga p(x)}dx\right)^{1/\ga}\le A\frac{1}{|Q|}\int_Q(tw(x))^{p(x)}dx,
$$
where $A>0$ will be determined later, then (\ref{maincond}) is trivial. Suppose that
\begin{equation}\label{aaa}
\frac{1}{|Q|}\int_Q(tw(x))^{p(x)}dx<\frac{1}{A}\left(\frac{1}{|Q|}\int_Q(tw(x))^{\ga p(x)}dx\right)^{1/\ga}.
\end{equation}

As in the proof of Lemma \ref{RH}, take $q=\frac{1+r(s-1)}{1+\ga(s-1)}$ and apply
H\"older's inequality with the exponents $q$ and $q'$. We obtain
\begin{eqnarray*}
&&\left(\frac{1}{|Q|}\int_Q(tw(x))^{\ga p(x)}dx\right)^{1/\ga}\\
&&\le \left(\frac{1}{|Q|}\int_Q(t^{\frac{1}{1+\e}}w(x))^{r p(x)}dx\right)^{\frac{1+\e}{r}}
\left(\frac{1}{|Q|}\int_Qw^{-\frac{1}{s-1}p(x)}dx\right)^{(s-1)\e}.
\end{eqnarray*}
From this, applying H\"older's inequality again along with (\ref{aaa}) yields
\begin{eqnarray}
&&\frac{1}{|Q|}\int_Q(t^{\frac{1}{1+\e}}w(x))^{p(x)}dx\label{intest}\\
&&\le \left(\frac{1}{|Q|}\int_Q(tw(x))^{p(x)}dx\right)^{\frac{1}{1+\e}}
\left(\frac{1}{|Q|}\int_Qw(x)^{p(x)}dx\right)^{\frac{\e}{1+\e}}\nonumber\\
&&\le \frac{1}{A^{\frac{1}{1+\e}}}\left(\frac{1}{|Q|}\int_Q(tw(x))^{\ga p(x)}dx\right)^{\frac{1}{\ga(1+\e)}}
\left(\frac{1}{|Q|}\int_Qw(x)^{p(x)}dx\right)^{\frac{\e}{1+\e}}\nonumber\\
&&\le \frac{1}{A^{\frac{1}{1+\e}}}[w(\cdot)^{p(\cdot)}]_{A_s}^{\frac{\e}{1+\e}}
\left(\frac{1}{|Q|}\int_Q(t^{\frac{1}{1+\e}}w(x))^{r p(x)}dx\right)^{1/r}.\nonumber
\end{eqnarray}

Further, from (\ref{aaa}) and from Lemma \ref{RH},
$t^{\frac{1}{1+\e}}\le \frac{1}{\|\chi_Q\|_{L^{p(\cdot)}_w}}$ (here we assume that $A\ge c$, where $c$ is the constant from Lemma \ref{RH}).
Hence, by Lemma \ref{strf},
$$
|Q|\left(\frac{1}{|Q|}\int_Q(t^{\frac{1}{1+\e}}w(x))^{rp(x)}dx\right)^{1/r}\le k\int_Q(t^{\frac{1}{1+\e}}w(x))^{p(x)}dx+b(Q),
$$
where $b(Q)$ is defined in the proof of Lemma \ref{strf}. 
Thus, taking $A=\max((2k)^{1+\e}[w(\cdot)^{p(\cdot)}]_{A_s}^{\e},c)$, where $c$ is the constant from Lemma \ref{RH}, and applying (\ref{intest}), we obtain
$$
|Q|\left(\frac{1}{|Q|}\int_Q(t^{\frac{1}{1+\e}}w(x))^{rp(x)}dx\right)^{1/r}\le 2b(Q),
$$
which implies
$$
|Q|\left(\frac{1}{|Q|}\int_Q(tw(x))^{rp(x)}dx\right)^{1/r}\le 2t^{\frac{\e}{1+\e}p_-}b(Q).
$$
This along with H\"older's inequality (since $\ga<r$) proves (\ref{maincond}).
\end{proof}

\section{Concluding remarks and open questions}
\subsection{About the assumption $w(\cdot)^{p(\cdot)}\in A_{\infty}$}
We start with the following question.
\begin{que}\label{rem}
Is it possible to remove completely the assumption $w(\cdot)^{p(\cdot)}\in A_{\infty}$ in Theorem \ref{mr}?
\end{que}

Several remarks related to this question are in order.
Denote by $LH({\mathbb R}^n)$ the class of exponents $p(\cdot)$ with $p_->1, p_+<\infty$ and such that
$$|p(x)-p(y)|\le \frac{c}{\log({\rm{e}}+1/|x-y|)}\,\,\,\,\text{and}\,\,\,\,
|p(x)-p_{\infty}|\le \frac{c}{\log({\rm{e}}+|x|)}$$ for all $x,y\in {\mathbb R}^n$, where $c>0$ and $p_{\infty}\ge 1$.
Also denote by $A_{p(\cdot)}$ the class of weights such that
$$\sup_Q|Q|^{-1}\|\chi_Q\|_{L^{p(\cdot)}_w}\|\chi_Q\|_{L^{p'(\cdot)}_{w^{-1}}}<\infty.$$

It was shown in \cite{CDH,CFN} that if $p(\cdot)\in LH({\mathbb R}^n)$, then $M$ is bounded on $L^{p(\cdot)}_w$ if and only if $w\in A_{p(\cdot)}$.
An important ingredient in the proof in~\cite{CFN} is the fact that if $p(\cdot)\in LH({\mathbb R}^n)$ and $w\in A_{p(\cdot)}$, then $w(\cdot)^{p(\cdot)}\in A_{\infty}$.
Since the boundedness of $M$ on $L^{p(\cdot)}_w$ implies the $A_{p(\cdot)}$ condition trivially, we see that if $p(\cdot)\in LH({\mathbb R}^n)$, then
the assumption that $w(\cdot)^{p(\cdot)}\in A_{\infty}$ in Theorem~\ref{mr} is superfluous. However, we do not know whether this assumption can be removed (or at least weakened) in general.

It is well known (see, e.g., \cite[Th. 3.16]{CF}) that if $p(\cdot)\in LH({\mathbb R}^n)$, then $M$ is bounded on $L^{p(\cdot)}$.
This fact raises the following questions.

\begin{que}\label{anq}
Suppose that $M$ is bounded on $L^{p(\cdot)}$ and $w\in A_{p(\cdot)}$.
Does this imply $w(\cdot)^{p(\cdot)}\in A_{\infty}$?
\end{que}

\begin{que}\label{anq1}
Is it possible to replace in Theorem \ref{mr} the assumption $w(\cdot)^{p(\cdot)}\in A_{\infty}$ by the boundedness of
$M$ on $L^{p(\cdot)}$?
\end{que}

Question \ref{anq1} is closely related to another open question stated in \cite{DH} and \cite[p. 275]{CF}: is it possible to deduce the equivalence
$M:L^{p(\cdot)}_w\to L^{p(\cdot)}_w\Leftrightarrow w\in A_{p(\cdot)}$
assuming only that $M$ is bounded on $L^{p(\cdot)}$?

\subsection{An application}
It is a well known principle that if $M$ is bounded on a BFS $X$ and on $X'$, then some other basic
operators in harmonic analysis are also bounded on $X$. Consider, for instance, a Calder\'on-Zygmund operator $T$.
By this we mean that $T$ is an $L^2$ bounded integral operator represented as
$$Tf(x)=\int_{{\mathbb R}^n}K(x,y)f(y)dy,\quad x\not\in\text{supp}\,f,$$
with kernel $K$ satisfying $|K(x,y)|\le \frac{c}{|x-y|^n}$ for all $x\not=y$, and for some $0<\d\le 1$,
$$|K(x,y)-K(x',y)|+|K(y,x)-K(y,x')|\le
c\frac{|x-x'|^{\d}}{|x-y|^{n+\d}},$$ whenever $|x-x'|<|x-y|/2$.
It was shown in \cite{L} that
$$\int_{{\mathbb R}^n}|Tf(x)g(x)|dx\le c\int_{{\mathbb R}^n}Mf(x)Mg(x)dx.$$
This estimate along with (\ref{hold}) and (\ref{repr}) implies that if $M$ is bounded on a BFS $X$ and on $X'$, then
$T$ is bounded on $X$. Hence, Theorem \ref{mr} yields the following corollary.

\begin{cor}\label{corol}
Let $p:{\mathbb R}^n\to [1,\infty)$ be a measurable function such that $p_->1$ and $p_+<\infty$. Let $w$ be a weight such that $w(\cdot)^{p(\cdot)}\in A_{\infty}$.
If $M$ is bounded on $L^{p(\cdot)}_w({\mathbb R}^n)$, then $T$ is bounded on $L^{p(\cdot)}_w$.
\end{cor}

As we have mentioned above, it was shown in \cite{CFN} that if $p(\cdot)\in LH({\mathbb R}^n)$ and $w\in A_{p(\cdot)}$, then $w(\cdot)^{p(\cdot)}\in A_{\infty}$ and $M$ is bounded on $L^{p(\cdot)}_w$.
Therefore, Corollary \ref{corol} implies the following less general result.

\begin{cor}\label{corol1}
If $p(\cdot)\in LH({\mathbb R}^n)$ and $w\in A_{p(\cdot)}$, then $T$ is bounded on $L^{p(\cdot)}_w$.
\end{cor}

Notice that a closely related result was very recently proved in \cite{INK}.

\vskip 3mm
{\bf Acknowledgement.}
I am grateful to Alexei Karlovich for valuable remarks on an earlier version of this paper.
Also I would like to thank the anonymous referee for detailed comments that improved the presentation.

\end{document}